\documentclass[a4paper,12pt]{amsart}
\usepackage[top=1.5in, bottom=1in, left=0.9in, right=0.9in]{geometry}

\usepackage[utf8]{inputenc}
\usepackage[all]{xy}
\usepackage{color}
\usepackage{hyperref}
\usepackage{enumitem}

\usepackage{amsmath, mathtools}
\usepackage{hyperref}
\hypersetup{colorlinks=true}
\usepackage{amssymb}
\usepackage{amsfonts}
\usepackage{graphicx}
\usepackage{mathrsfs}
\usepackage{amscd}
\usepackage[all]{xy}
\usepackage{hyperref}
\usepackage{amsthm}
\usepackage{cleveref}
\usepackage{verbatim}
\usepackage{amscd}
\usepackage{mathtools}
\usepackage{comment}

\newtheorem{thm}{\bf Theorem}[section]

\newtheorem{prop}[thm]{\bf Proposition}

\newtheorem{cor}[thm]{\bf Corollary}
\theoremstyle{definition}
\newtheorem{definition}[thm]{\bf Definition}
\theoremstyle{remark}
\newtheorem{remark}[thm]{\bf Remark}

\newtheorem{example}[thm]{\bf Example}
\numberwithin{equation}{section}

\DeclareMathOperator{\depth}{{depth}}

\DeclareMathOperator{\Tor}{{Tor}}

\DeclareMathOperator{\soc}{{{soc}}}

\DeclareMathOperator{\projdim}{{{projdim}}}

\DeclareMathOperator{\len}{{length}}
\DeclareMathOperator{\BI}{{BI}}
\DeclareMathOperator{\Burch}{{Burch}}

\DeclareMathOperator{\syz}{syz}

\def\f0{\mathbf{0}}

\def\fn{\mathfrak{n}}

\def\fm{\mathfrak{m}}

\def \PP{\mathbb P}

\def \ZZ{\mathbb Z}

\def \L{\mathcal L}

\def\CL{\mathcal{CL}}

\def\lin{{\rm lin}}
\def\Lin{{\rm LinSp}}

\begin{document}

\title[Linear syzygies]{Burch index, summands of syzygies and  linearity in resolutions}

\author[Hailong Dao]{Hailong Dao}
\address{Hailong Dao\\ Department of Mathematics \\ University of Kansas\\405 Snow Hall, 1460 Jayhawk Blvd.\\ Lawrence, KS 66045}
\email{hdao@ku.edu}

\author[David Eisenbud]{David Eisenbud}
\address{Department of Mathematics, University of California at Berkeley and the Mathematical
Sciences Research Institute, Berkeley, CA 94720, USA}
\email{de@msri.org}

 \begin{abstract}
 The Burch index is a new invariant of a local ring $R$ whose positivity implies a kind of linearity in resolutions of $R$-modules. We show that if $R$ has depth zero and Burch index at least $2$, then any non-free 7th $R$-syzygy contains the residue field as a direct summand. We compute the Burch index in various cases of interest.  \end{abstract}

\maketitle

\section*{Introduction}

Let $(R,\fm, k)$ be a commutative, Noetherian local ring. Little is known about the
matrices that can occur in the minimal free resolutions of a finitely generated $R$-module of infinite projective dimension. In this paper we study
conditions on $R$, expressed by an invariant that we call the \emph{Burch index}, that imply that these matrices contain many elements outside the square of the
maximal ideal, or that the syzygies of $M$ contain $k$ as a direct summand.

The second condition is generally stronger than the first: by the main theorem of Gulliksen~\cite{G},
the Koszul complex of the maximal ideal $\fm$ of $R$ is a tensor factor of the minimal $R$-free
resolution of $k$, so, when $\dim_k(\fm/\fm^2) \geq 2$,  the entries of each matrix in the minimal free resolution of $k$ generate $\fm$.
It follows that the same holds for any module with $k$ as a direct summand.

To exclude the phenomena associated with modules of finite projective dimension, it is useful to consider the case of rings of depth 0, and the Burch index of an arbitrary Noetherian local ring is defined in Section~\ref{SectionDef} by reducing to this setting.
Suppose that $R$ is a depth 0 local ring with residue field $k$, and that $R = S/I$, where $(S,\fn, k)$ is a  regular local ring. Note that  $\fn^2 \subset I\fn:(I:\fn)\subset \fn$. We define the \emph{Burch index} of $R$ (relative to $S$) to be 
$$
\Burch_S(R) = \dim_k \frac{\fn}{I\fn:(I:\fn)}.
$$
Intuitively this measures the minimal generators of $I$ that are multiples of elements
in the socle of $S/I$.  Our definition is inspired by (\cite{DKT}) in which a local ring $(R,\fm, k)$ is said to be
a \emph{Burch ring} if, in our terms, it has Burch index $\geq 1$. It is proved there that $R$ is Burch if and only
if  the second syzygy of $k$ has $k$ as a direct summand. See that paper for background. 

It turns out that
if $R$ has Burch index $\geq 2$ then similar things are true, not just for $k$, but for \emph{every} module of infinite projective
dimension:

\begin{thm}\label{maint1}
Suppose that $(R,\fm,k)$  is a local ring of depth 0 with $\Burch(R) \geq 2$. If $M$  is an $R$-module
that is not free, then   the entries of each matrix in the minimal free resolution of $\syz^R_5(M)$ generate $\fm$. Furthermore,
 $k$ is a direct summand of
$\syz_n^R(M)$ for all $n\geq 7$.
\end{thm}

 See Theorem \ref{main-theo} for a stronger but more complicated result.

Theorem~\ref{maint1} seems to be new even in the case of the rings $R=k[x_1,\dots,x_d]/(x_1,\dots,x_d)^n$ ($n\geq 2$), which have Burch index $d$. Theorem \ref{maint1} implies immediately that over a ring of Burch index at least $2$, each module  has either finite projective dimension or maximal curvature (that is, its Poincare series has the same radius of convergence as that of the residue field). See \cite{A} for a survey of the vast literature on this subject. Also, any high enough (non-free) syzygy $M$ over such ring can be used to test for finite projective dimension, in the sense that $\Tor_i^R(M,N)=0$ for a single $i>0$ forces $N$ to have finite projective dimension. 

If $M$ is a finitely generated $R$-module, we write $J(M)$ for the ideal generated by the entries
of a minimal presentation matrix for $M$, and $\syz_n^R(M)$ for the minimal $n$-th $R$-syzygy module of $M$. Without assuming
that $R$ has depth 0 (and using the definition in Section~\ref{SectionDef} we deduce that over
rings of Burch index $\geq 2$ the presentation matrices of high syzygies contain many linear forms:

\begin{cor}\label{main-cor}
Let $(R,\fm,k)$ be a local ring and suppose $\Burch(R)\geq 2$. Let $M$ be any finitely generated module over $R$ which has infinite projective dimension. 
For $n\geq 5+\depth R$, 
$$
\dim_k  \frac{\fm}{J(\syz^R_n M)+\fm^2} \leq \depth R.
$$
\end{cor}

In Theorem \ref{I-lin} we draw a connection between the Burch index of an ideal $I$ in a regular local ring $S$ and the linear entries of matrices in its $S$-free resolution. Our results there suggest a  connection between linearity of resolution of $I$ over $S$ and the resolution of modules over $S/I$. 

We conclude the paper by computing the Burch index for several special classes of ideals: ideals in a regular local ring of dimension $2$,  ideals of general sets of points in the projective plane,  ideals with almost linear resolution, and ideals of certain fibre products. 
We give examples of monomial ideals to illustrate our results.

The main results of this paper were suggested by extensive  computations using the computer algebra
system Macaulay2~\cite{M2}. We wish to acknowledge the 
importance of this tool, primarily produced and constantly improved by Dan Grayson and Mike Stillman.

\section{Definitions and preliminary  results}\label{SectionDef}

Throughout this paper, all rings are local and Noetherian and modules are finitely generated. We regard positively graded algebras over a field  as local as well, with the results applying to graded modules only.

\begin{definition}[\emph{relative Burch index}]\label{def1}
Let $(S,\fn, k)$ be a local ring and $I\subseteq S$ an ideal.
We define the \emph{Burch ideal} of an ideal $I\subset S$ to be $\BI_S(I) := I\fn:(I:\fn)$. 

If $I=0$ or $S = S/I$ has positive depth,  we set $\Burch_S(I)=0.$ Otherwise $\fn^2\subseteq \BI_S(I)\subseteq \fn$ and we  set
$$\Burch_S(I) = \dim_k(\fn/\BI_S(I)).$$
\end{definition}

Our definition of the (absolute) Burch index of a local ring uses the Cohen presentation, a residue field extension, and reduction modulo  regular sequences to reduce to the case of an ideal of codepth 0 in a regular local ring. By a \emph{regular presentation} of $R$, we mean a surjective local map $(S,\fn,k)\to (R,\fm,k)$ where $S$ is regular. Such a presentation is called minimal if $\dim_k \fn/\fn^2= \dim_k \fm/\fm^2$. 

\begin{definition}\label{def2}
Let $(R,\fm,k)$ be a local ring and let $\hat R$ denote the $\fm$-adic completion of $R$. 

\begin{enumerate}
\item If $\depth R=0$, let $\hat R=S/I$ be a minimal regular Cohen presentation of $\hat R$. Define $\BI(R)$ to be $\BI_S(I)\hat R\cap R$ and   $\Burch(R)$ to be 
$\Burch_S(I)$. These are well-defined by Theorem~\ref{cohen}. 
\item Assume  $\depth R>0$. Let $R'=R[z]_{(\fm R[z])}$. We define $\Burch(R)$ to be the maximum, over all maximal regular
sequences $(\underline x)$ in $R'$, of $\{\Burch(R'/\underline{x}R')\}$
\end{enumerate}
\end{definition}

Recall that a ring $R$ with $\Burch(R)>0$ is called a Burch ring in \cite{DKT}.

\begin{thm}\label{cohen}
Let $R$ be a local ring of depth $0$ and consider a regular  presentation $R=S/I$. Then the $R$-ideal $\BI_S(I)R$ and the number $\Burch_S(I)$ are independent of the presentation of $R$. 
\end{thm}

\begin{proof}
We can assume $R$ is complete and not a field. Consider any two presentations of $R$, $S_1 \twoheadrightarrow R$ and $S_2\twoheadrightarrow R$. Then $S=S_1\times_RS_2$ is also complete local, and thus by the Cohen Structure Theorem it is a homomorphic image of another regular local ring $T$. Hence it is enough to prove the equality of our definitions for the pair $(T,S_1)$ and since $T$ maps onto $S_1$, by induction we reduce to the case where our presentations are coming from $(T, \fn_T) \twoheadrightarrow (S, \fn_S)=T/(x) \twoheadrightarrow R$ where $x\in \fn_T-(\fn_T)^2$.

Suppose $R=S/I$. Let $(J,\fm)$ be lifts of $(I,\fn_S)$ to $T$ respectively. By that we mean: the minimal generators of $J$, together with $x$, generate the preimage of $I$ in $T$ minimally, and similarly for $\fm, \fn_S$.  Write $\fn$ for $\fn_T$.

We need to show that $$J\fm+(x): [(J,x):\fn] = (J+(x))\fn: [(J,x):\fn] $$ (the LHS is $\BI_S(I)$ and the RHS is $\BI_T(I)$). Observe that $(J+(x))\fn = J\fm+x\fn$. Let $A=T/J\fm$ and (slightly abuse notations) $x,\fn, H$ be the images of $x,\fn, (J,x):\fn$  in $A$. We are reduced to showing $(x):H =x\fn:H$. By assumption $R$ is not a field so $H\subset \fn$. Thus, if $a\in  (x):H$, then $aH \subset (x)\cap \fn^2 = (x)\fn$, and we are done. 
\end{proof}

The following result shows in particular that to compute the Burch index for rings of positive depth one must use  regular sequences consisting entirely of elements not in the square of the maximal ideal, that is, \emph{linear} regular sequences:

\begin{prop}\label{Tor_ind}
 Let $S$ be regular and $I,J$ be ideals such that $IJ=I\cap J$ (equivalently, $\Tor^S_{i>0}(S/I,S/J)=0$). If $S/I,S/J$ are both singular, then $\Burch(S/(I+J))=0$. In particular, if $R=S/(I+f)$, where $f$ is a nonzerodivisors in the square of the maximal ideal of $S/I$, then $\Burch(R)=0$. 
\end{prop}

\begin{proof}
Let $R=S/(I+J)$. Then $R$ is not $\Tor$-friendly, in the language of \cite{AINS}. Namely, there are $R$ modules $M,N$ of infinite projective dimension such that  $\Tor_{i>0}^R(M,N)=0$. On the other hand, Burch rings are $\Tor$-friendly, by \cite[Theorem 1.4]{DKT}. 
\end{proof}

At this point, we have enough to give basic estimates and precise values for the Burch index of certain well known rings. See also Section~\ref{examplesSection}.

\begin{prop}\label{basic}
Let $(R, \fm,k)$ be a local ring. 
\begin{enumerate}
\item $\Burch(R)\leq \dim_k(\fm/\fm^2)-\depth R$. 
\item If $R = S/J\fn$ where $(S,\fn)$ is a regular local ring and $0\neq J\subset \fn$, then $\Burch(R)=\dim S$. 
\item If $R$ is a Cohen-Macaulay ring of minimal multiplicity, then $\Burch(R)= \dim_k(\fm/\fm^2)-\depth R$.
\item If $R$ is a singular hypersurface, then $\Burch(R)=1$.
\end{enumerate}
\end{prop} 

\begin{proof}

$(1)$ The inequality is obvious if $\depth R=0$. When $\depth R>0$, by \ref{Tor_ind}, to compute $\Burch(R)$, we just need to use a regular sequence in $\fm-\fm^2$, thus the statement reduces to the $\depth 0$ case, whence the assertion is obvious. 

$(2)$ Clearly $\depth S/J\fn=0$. We have $\BI_S(J\fn)=  J\fn^2:(J\fn:\fn) \subseteq J\fn^2:J \subseteq  \fn^2 \subseteq  \BI_S(J\fn)$.

$(3)$ We may assume $R$ is Artinian, and thus has the presentation $S/\fn^2$ for some regular local ring $(S, \fn)$, so we can apply \noindent $(2)$ The inequality is obvious if $\depth R=0$. When $\depth R>0$, by \ref{Tor_ind}, to compute $\Burch(R)$, we must use a regular sequence in $\fm-\fm^2$, so the statement reduces to the $\depth 0$ case.

$(4)$ By the argument  above, the computation reduces to the Artinian case, so we can assume $R=S/(t^a)$ for some DVR $(S,(t))$ and some $a>1$, for which the claim is clear. 

\end{proof}

\section{Burch index and linearity}

Let $(S,\fn,k)$ be a local ring, let $I\subseteq S$ an ideal and set $R=S/I$. Our main results link high values of Burch index to  linear elements in the $S$-free resolution of $R$  and the $R$-free resolutions of $R$-modules.  

\begin{definition}\label{def-lin}
Recall that we write $J(M)$ for the ideal generated by the entries
of a minimal presentation matrix for a finitely generated module $M$. We set 
$$
\begin{aligned}
 \Lin_n^S(M) &:= \frac{J(\syz^S_{n-1}M) +\fn^2}{\fn^2}\\
\lin_n^S(M) &:= \dim_k \Lin_n^S(M), 
\end{aligned}
$$
and think of $\Lin_n^S(M)$ as the vector space of linear forms in the $n$-th matrix in a minimal free resolution of $M$.
\end{definition}

In the case where the ring $S$ has depth $\geq 2$, the following result gives another characterization
 of $BI_S(I)$.
 
\begin{prop}\label{bm}
Let $t\in \fn$ be a nonzerodivisor. If $t\notin \BI_S(I)$ then $k$ is a direct summand of $I/tI$. If $\depth S\geq 2$ then
the converse also holds.
\end{prop}

\begin{proof}
Since $t\in \fn$ is regular, $tI:\fn \subseteq tI:(t) = I$, and thus 
$$
\soc(I/tI) = ((tI:\fn)\cap I)/tI = (tI:\fn)/tI.
$$
Thus $I/tI$ has a $k$-summand  if and only if  $(tI):\fn\not \subseteq \fn I$. 

If $t\notin \BI_S(I) = \fn I : (I:\fn)$, then $t(I:\fn)\not\subseteq \fn I$, and since $t(I:\fn) \subseteq (tI):\fn$,
$I/tI$ has a $k$-summand.

If $\depth S\geq 2$ then  $\fn$ contains an element that is a nonzerodivisor on $S/tS$, so $(tI):\fn=t(I:\fn)$, and the previous argument can be reversed. 
\end{proof}

\begin{remark}\label{mod-t}
Let $(S,\fm)$ be a local ring and $M$ an $S$-module.  If $t\in S$ is a regular element on $M$ and $S$, then for each $i>0$, 
$$
J(\syz_{i-1}^{S/tS}(M/tM)) = \bigl(J(\syz_{i-1}^S(M))+tS\bigr)/tS.
$$
Thus $ \lin_i^{S/tS}(M/tM) = \lin_i^S(M) $ if the image of $t$ in $\fn/\fn^2$ is outside of $\Lin^S_i(M)$,
and otherwise $ \lin_i^{S/tS}(M/tM) = \lin_i^S(M) -1$.
\end{remark}

The matrices in the $S$-free resolution of ideals with positive Burch have many linear entries:

\begin{thm}\label{I-lin}\label{cor-lin}
 Let $(S, \fn,k)$ be a local ring of embedding dimension $n$, and let $I\subset \fn$ be an ideal. If $0<i< \projdim I$, then:
\begin{enumerate}
\item If $\Burch_SI \geq 1$ and $S$ has positive depth, then  $\lin_i^S(I) \geq n-1$.
\item If $\Burch_SI \geq 2$ and $\dim (P+\fn^2)/\fn^2 \leq n-2$ for each associated prime $P$  of $S$ then  $\lin_i^S(I) = n$.
\end{enumerate}
\end{thm}

\begin{proof} 
 If $t\notin \BI_S(I)$ is a nonzerodivisor, then
by \ref{bm} the module $I/tI$  has a $k$-summand, so $\Lin_i^{S/tS}(I/tI)= \fn/(\fn^2+(t))$, and thus $\Lin_i^S(I)+\langle t\rangle=\fn/\fn^2$ by Remark $\ref{mod-t}$, hence 
$\lin_i^S(I)\geq n-1$.

If $\Burch_SI \geq 1$ and $S$ has positive depth then, by prime avoidance, we can choose such an element $t$,
proving $(1)$.

If $\Burch_SI \geq 2$ then let the associated primes of $S$ be $P_1,\dots, P_a$. Let $V_0= (\Burch_SI) +\fn^2)/\fn^2$   and  $V_i= (P_i+\fn^2)/\fn^2$. For each nonzero vector $t\in \fn/\fn^2$ outside of $\cup_{j=0}^a V_j$ we must have $\langle t\rangle+\Lin_i^S(I)=\fn/\fn^2$ and each $V_j$ has codimension at least $2$,  so item (2) follows.
\end{proof}

%
%

\begin{example}
We cannot  deduce $\Burch(I)$ from the dimensions $\lin_i^S(I)$. For example if $I=(x^3,y^3,z^3, x^2y,y^2z,z^2x)\subset S=k[[x,y,z]]$, then $\lin_1(I)=\lin_2(I)=3$, but $\Burch_S(I)=0$.  However, if $I$ is an ideal of projective dimension 1, then Corollary~\ref{proj1} says that $\Burch(S/I)= \min\{\lin_1^S(I), 2\}$. 
\end{example}

\section{Burch index and $k$-direct summands}

As we explained in the introduction, the condition that a module $M$ has $k$ as a direct summand
is stronger than the condition that $\lin_i(M)$ is the maximal ideal. Our main theorem strengthens Theorem~\ref{maint1}.

\begin{thm}\label{main-theo}
Let $(R,\fm, k)$ be a local ring of depth $0$ and embedding dimension $\geq 2$. For every non-free $R$-module $M$:
\begin{enumerate}

\item If $\Burch(R)\geq 2$ then $k$ is a direct summand of $\syz^R_i(M)$ for some $i\leq 5$ and for all $i\geq 7$. 

\item If $\Burch(R)\geq 1$, and $k$ is a direct summand of  $\syz^R_s(\BI(R))$ for some $s\geq 1$, then $k$ is a direct summand of $\syz^R_i(M)$ for some $i\leq s+4$ and for all $i\geq s+6$.
\end{enumerate}
Thus, if $R$ is Burch, then $k$ is a direct summand of all high syzygies of non-free $R$-modules if and only if $k$ is a direct summand of some $R$-syzygy of $\BI(R)$. 
\end{thm}

 \begin{remark}
 The case of depth 0 and embedding dimension 1 is different, but easy to analyze completely: in this case 
  $\hat R = S/t^n$ for some discrete valuation ring $(S,t)$, so $\BI_S(R) = t$. Any indecomposable $R$-module has the form
  $M = R/t^m$ and the syzygies of $M$ have $k$ as direct summand in alternate steps iff $m= 1$ or $m=n-1$.
\end{remark}

The following Propositions will be used in the proof of Theorem \ref{main-theo}. Recall the convention that the ideal of entries of a matrix with no columns is $R$. We first show that the ideals $J(\syz^R_{s-1}(M)) + J(\syz_s^R(M))$, for any $R$-module $M$, are an obstruction
to socle summands:

\begin{prop}\label{prop1}
Let $M$ be a module over the local ring $(R,\fm)$. If 
$$
\bigl(J(\syz^R_{s-1}(M)) + J(\syz_{s}^R(M))\bigr) \subseteq K \subsetneq \fm 
$$
 for some ideal $K$ of $R$ then $k$ is not a summand of $\syz_p(K)$ for any $1\leq p\leq s-2$.
\end{prop}

\begin{proof}
Let $d_i$ be the $i$-th differential in a minimal free resolution of $M$. The hypothesis is equivalent to the statement
that $R/K \otimes d_{s} = 0 = R/K\otimes d_{s+1}$. This implies that $\Tor_{s}^R(M,R/K)$ is a nonzero free $R/K$-module,
and thus has no $k$-summand. But if
$k$ were a direct summand of $\syz_p^R(K)$ with $1\leq p\leq s-2$ then 
$$
\Tor_{s}^R(M,R/K)= \Tor_{s-1}^R(M, K) = \Tor_{s-1-p}^R(M, \syz_p^R(K))
$$
 would have a $k$-summand, a contradiction.
\end{proof}

%
%
%

We will use~\cite[Lemma 4.1]{KV}, which was generalized in~\cite[Proposition 3.4]{DKT}. For the reader's convenience we give the statement and proof:

\begin{prop}\label{dkt}
Let $(S,\fn,k)$ be a local ring, $I\subset S$ an ideal and $R=S/I$. Let $A:F\to G$ be a map of free $S$-modules, and let $J$ be the ideal generated by the entries of a matrix representing $A$. If  $J \subset \fn$ but $J \nsubseteq \BI_S(I)$, then $k$ is a direct summand of $\ker R\otimes A$.   
\end{prop}

\begin{proof} 
The hypothesis means that there is are elements $s\in I:\fn$ and $t\in J$ such that $st \notin \fn I$. Thus
for some minimal generator $e$ of $F$,  $A(se) \notin \fn IG$.

Write $\overline{\phantom{A}}$ for reduction mod $I$.
 We claim that $\overline{se}$ is a socle element and a direct
summand of $\ker \overline A$.
Clearly 
$$
se \notin A^{-1}(\fn IG) \supset \fn A^{-1}(IG),
$$
so $\overline{se} \notin \fn \ker \overline A$.
On the other hand, $\overline{se}\in \soc \overline F$, and since
$J\subset \fn$ we have $\overline{se} \in \ker \overline A$.  
\end{proof}

%
%
%
%

\begin{proof}[Proof of Theorem~\ref{main-theo}]
 Let $\L_2(R)$ be the set of ideals of $R$ with $\len(R/L)=2$. 

\noindent $(1)$ Since $\len(S/\BI_S(I))\geq 3$, the ideal $\BI_S(I)$ does not contain the preimage in $S$ of
any $L\in \L_2(R)$.  By Proposition~\ref{dkt} the module $k$ is a direct summand of $\syz_2^R(R/L)=\syz_1^R(L)$, so Proposition~\ref{prop1} with $s=3$ gives $J(\syz_2^RM)+J(\syz_3^R(M))= \fm$, and thus not both of 
$J(\syz_2^R(M))$ and $J(\syz_3^R(M))$ can be contained in  $\BI_S(I)$. By Proposition~\ref{dkt}, $k$ is a direct summand of  $\syz_{4}^R(M)$ or $ \syz_{5}^R(M)$, as claimed. 

Since $R$ has embedding dimension $\geq 2$ we have $J(\syz_i^R(k)) = \fm$ for every $i\geq 1$,
and thus $J(\syz_i^R(M)) = \fm$ for every $i\geq 6$.
Using Proposition~\ref{dkt} again we see that $k$ is a summand of $\syz_i^R(M)$ for each $i\geq 7$. 

\noindent $(2)$ In this case $\len S/\BI_S(R) = 2$, so the preimage in $S$ of an ideal $L\in \L_2(R)$ 
contains $\BI_S(R)$ if and only if it is equal to $\BI_S(R)$. Thus the reasoning of part $(1)$ applies to every $L\in \CL(R)$ except for $\BI(R)$ itself, to which we can apply the hypothesis of part $(2)$. 
\end{proof}

\begin{example}
We can see that theorem \ref{main-theo}  is sharp by taking $S=k[[a,b]], I=(a,b^2)^2$. Set $R=S/I$ and let $M=R/(a,b^2)$.  Then $\BI_S(I)=(a,b^2), \Burch(R)=1$ and one can see that $\syz_1^R(M)\cong M^{\oplus 2}$. Thus $k$ is not a direct summand of any $R$-syzygy of $M$. 
\end{example}

\begin{example}\label{sharpness of 5}
The bound of $5$ in Theorem \ref{main-theo} cannot be improved. Let $k=\ZZ/(101)$ and $S=k[x,y,z]$. Let $I=(x^3,y^3,z^3)(x,y,z)$ and $R=S/I$. Let $M=R/(xyz)$. Macaulay2 \cite{M2} tells us that $\Burch(R)=3$ and $\syz_5^R(M)$ is the first syzygy of $M$ that has a $k$-summand. 
\end{example}


\begin{cor}
Let $(R, \fm,k)$ be a local ring of embedding dimension $n$ and $\depth R=t$. If $\Burch(R)\geq 2$, then $\lin_i(M)\geq n-t$ for any module $M$ of infinite projective dimension and  each $i\geq t+5$. 
\end{cor}

\begin{proof}
Suppose $\depth R=t$ and $\Burch(R)\geq 2$. If $t=0$, we use part $(1)$ of \ref{main-theo}.

The case $t>0$ follows from the definition of Burch index, Remark \ref{mod-t} and induction. 
\end{proof}

\section{Some computations of the Burch index}\label{examplesSection}

\begin{thm}\label{dim2}\label{proj1}
Let $(S, \fm)$ be a regular local ring. Let $I\subset \fm^2$ be an ideal of projective dimension $1$. Then $\Burch(S/I)= \min\{\lin_1^S(I), 2\}$. 
\end{thm}

\begin{proof} Factoring out an appropriate regular sequence to reduce to the depth 0 case, using  \ref{mod-t} and induction, we may assume $\dim S = 2$.

Let $J=J(I)$, the ideal generated by the entries of the Hilbert-Burch matrix of $I$. Consider $t\in \fm-\fm^2$. By \ref{bm}, $t\in \BI(I)$ if and only if $k$ is not  a summand of $I/tI$. The ring $S':= S/tS$ is a discrete valuation ring; write $s$ for its parameter. 

The $S'$-module $I/tI$ is a direct sum of copies of $S'$ and cyclic modules of the form  $S'/\fm^aS'$ for various $a\geq 1$. Thus, $k$ is a summand of $I/tI$  if and only if $\Lin_1^{S/tS}(I/tI) \nsubseteq  \fm^2(S/tS)$ or, equivalently,
$(t)+ J \nsubseteq (t)+\fm^2$. From this observation, $\BI_S(I)$ is equal to:  $\fm$ if $\lin_1(I)=0$, $J$ if $\lin_1(I)=1$, and $\fm^2$ if $\lin_1(I)=2$, and our statement follows.    
\end{proof}

Here we apply the above to a geometric case:
\begin{thm}
Let $R$ be the homogenous coordinating ring of $d>1$ generic points in $\PP^2_k$ where $k$ is an algebraically closed field. Then $\Burch(R)$ is 
\begin{itemize}
\item $1$ if $d=2$.
\item $0$ if $d=\binom{2s+1}{2}+s$ for some integer $s>0$. 
\item $2$ in all other cases. 
\end{itemize}
\end{thm}
\begin{proof}
Write $R=S/I$ where $S=k[x,y,z]$. Since $\projdim_S I=1$, by \ref{proj1} we just need to compute $\lin_1^S(I)$. The degrees of elements in the Hilbert-Burch matrix $H$ of $I$ are well-known, see for instance \cite[Exercises 12, 13, pp 50]{E}. First, write $d$ (uniquely) as $d= \binom{s+t+1}{2}+s$, where $s,t$ are non-negative integers. The linear forms in $H$ form a submatrix $L$ of size $(t+1)\times(t-s)$ if $t\geq s$, or $(s-t)\times s$ if $s\geq t$. Since we are in the generic situation, the space of linear forms has dimension $\min\{3,n\}$ where $n$ is the number of entries in $L$, thus $\lin_1(I)= \min\{2,n\}$. It is easy to see that $n=1$ if and only if $s=1, t=0$, $n=0$ if and only if $s=t$, and the assertions  are now clear.  \end{proof}

\begin{prop}\label{deg_soc}
Let $(S, \fm)$ be a polynomial ring and $I\subset \fm^2$ a homogenous ideal of depth $1$. Suppose that the minimal generating degree of $I:\fm$ is less than the minimal generating degree of $I$.  Then $\Burch(S/I) = \dim S$. 
\end{prop}

\begin{proof}
If $g$ is the minimal generating degree of $I:\fm$, then no linear forms can multiply a generator with degree $g$ to $I\fm$, which starts in degree at least $g+2$, thus $\BI_S(I) = \fm^2$, as required.
\end{proof}

If $I$ is a homogeneous ideal in a polynomial ring in $p$ variables, then we say that $I$ the resolution 
of $I$ is \emph{almost linear} if it is linear for $p-1$ steps where $p=\projdim_S I$.

\begin{cor}\label{almost_lin}
Let $(S, \fm)$ be a polynomial ring and $I\subset \fm^2$ a homogeneous ideal of projective dimension $p$ ideal with almost linear resolution. Suppose that $\lin_p(I)>0$. Then 
 $\Burch(S/I) = \dim S-\depth S/I$. 
\end{cor}

\begin{proof}
We can assume $\depth S/I=0$. Let $g$ be the smallest degree of a generator of $I$. Then the resolution assumption tells us that the  socle of $I$ has a minimal generator in degree $g-1$, and \ref{deg_soc} applies.  

\end{proof}

\begin{example}
Corollary \ref{almost_lin} cannot be improved. There are Artinian Gorenstein ideals with almost linear resolution in each generating degree $d>1$ and number of variables $n>2$, see for instance \cite[Example 3.2]{DE}. Since the quotients are Gorenstein and not hypersurfaces, the Burch index is $0$ (the syzygies of $k$ will never have $k$ summands, as they are indecomposable). 
\end{example}

\begin{cor}
Let $(S, \fm)$ be a polynomial ring and $I\subset \fm^2$  an ideal with linear resolution. Then 
$\Burch(S/I) = \dim S-\depth S/I$. 
\end{cor}

\begin{example}[Fibre product] Let $A=k[[x_1,\dots,x_a]]/I$ and $B=k[[y_1,\dots, y_b]]/J$ be minimal presentations of local rings of depth $0$. Let $S=A\times_kB = k[[x_1,\dots, x_a, y_1, \dots, y_b]]/(I+J+(x_1,\dots,x_a)(y_1,\dots,y_b))$. Then it can be shown that $\Burch(S) = \Burch(A)+\Burch(B)$. 
\end{example}



\end{document}